\documentclass[12pt, twoside]{article}

\usepackage{times}
\usepackage{enumerate}

\usepackage{amsmath,amsthm,amsfonts,amssymb}
\usepackage{graphicx}

\usepackage[usenames]{color}
\usepackage{hyperref}
\usepackage{soul}

\usepackage{amsfonts, amsmath, amssymb, amscd, amsthm}
\usepackage{graphicx}

\newtheorem{theorem}{Theorem}[section]
\newtheorem{lemma}[theorem]{Lemma}
\newtheorem{cor}[theorem]{Corollary}
\newtheorem{prop}[theorem]{Proposition}

\theoremstyle{definition}
\newtheorem{defn}[theorem]{Definition}
\theoremstyle{remark}

\theoremstyle{definition}

\newtheorem{definition}[theorem]{Definition}

\newcommand{\G }{\Gamma (G, X\cup \mathcal H)}

\newcommand{\Hl }{\{ H_\lambda , \lambda \in \Lambda \} }

\newfont{\eufm}{eufm10}

\renewcommand{\L}{{\mathcal{L}}}

\newcommand{\Th}{{\mathrm{Th}}}

\newcommand{\ucl}{{\mathbf{Ucl}}}

\newenvironment{blue}{\color{blue}}{}

\newcounter{comcount}

\title{Limits of relatively hyperbolic groups and Lyndon's completions }

\author{Olga Kharlampovich and Alexei Myasnikov \footnote{both authors supported by NSERC grant}\\
McGill University}

\date{March 17, 2010}

\begin{document}
\maketitle

\begin{abstract}
In this paper we describe finitely generated groups $H$ universally
equivalent (with constants from $G$ in the language) to a given
torsion-free relatively hyperbolic group $G$ with free abelian
parabolics. It turns out that, as in the free group case, the group
$H$ embeds into the Lyndon's completion $G^{\mathbb{Z}[t]}$ of the
group $G$, or, equivalently, $H$ embeds into a group obtained from
$G$ by finitely many extensions of centralizers. Conversely, every
subgroup of $G^{\mathbb{Z}[t]}$ containing $G$ is universally
equivalent to $G$. Since finitely generated groups
universally equivalent to $G$  are precisely the finitely generated
groups discriminated by $G$  the result above gives a description of
finitely generated groups discriminated by $G$. Moreover, these groups are exactly the coordinate groups of irreducible algebraic sets over $G$.

\end{abstract}

\tableofcontents

\section{Introduction}
 Denote by $\mathcal G$ the class of all non-abelian torsion-free
relatively hyperbolic groups with free abelian parabolics. In this
paper we describe finitely generated groups that have the same
universal theory as a given group $G\in\mathcal G$ (with constants
from $G$ in the language). We say that they are universally
equivalent  to $G$. These groups are central to the study of logic
and algebraic geometry of $G$. They are coordinate groups  of
irreducible algebraic varieties over $G$. It turns out that, as in
the case when $G$ is a non-abelian free group \cite{KMIrc}, a
finitely generated group $H$ universally equivalent to $G$ embeds
into the Lyndon's completion $G^{\mathbb{Z}[t]}$ of the group $G$,
or, equivalently, $H$ embeds into a group obtained from $G$ by
finitely many extensions of centralizers. Conversely, every subgroup
of $G^{\mathbb{Z}[t]}$ containing $G$ is universally equivalent to
$G$ \cite{BMR2}. Let $H$ and $K$ be $G$-groups (contain $G$ as a
subgroup). We say that a family of $G$-homomorphisms (homomorphisms
identical on $G$) ${\mathcal F} \subset Hom_G(H,K)$ {\it separates}
[{\it discriminates}] $H$ into $K$ if for every non-trivial element
$h \in H$ [every finite set of non-trivial elements $H_0 \subset H$]
there exists $\phi \in {\mathcal F}$ such that $h^\phi \neq 1$
[$h^\phi \neq 1$ for every $h \in H_0$]. In this case we say that
$H$ is $G$-{\it separated} [$G$-{\it discriminated}] by $K$.
Sometimes we do not mention $G$ and simply say that $H$ is separated
[discriminated] by $K$. In the event when $K$ is a free group we say
that $H$ is  {\it freely separated} [{\it freely discriminated}].
 Since finitely generated groups
universally equivalent to $G$ are precisely the finitely generated
groups discriminated by $G$ (\cite{BMR1}, \cite{MR2}), the result
above gives a description of finitely generated groups discriminated
by $G$ or fully residually $G$ groups. These groups are exactly the coordinate groups of irreducible algebraic sets over $G$. Therefore we obtain a complete description of irreducible algebraic sets over $G$. Our proof uses
the results of \cite{Groves2} and \cite{RHG}, \cite{ESBG}.

\subsection{Algebraic sets}

Let $G$ be  a group generated by $A$,
$F(X)$ - free group on  $X = \{x_1, x_2, \ldots x_n\}$.
 A { system  of equations} $S(X,A) = 1$  in variables $X$ and
coefficients from $G$ can be viewed as a subset of $G \ast F(X)$.
A solution of $ S(X,A) = 1$ in  $G$ is a tuple $(g_1,
\ldots, g_n) \in G^n$ such that $S(g_1, \ldots, g_n) =1$ in $G$.
$V_G(S)$, the set of all solutions of  $ S = 1$ in $G$,
is called  an {algebraic set } defined by $S$.

 The maximal subset $R(S) \subseteq G \ast F(X)$ with
 $$V_G(R(S)) = V_G(S)$$ is the { radical }  of $S = 1$ in $G$.
The quotient group
$$G_{R(S)}=G[X]/R(S)$$
is the { coordinate group} of $S= 1$.

The following conditions are equivalent
 \begin{itemize}
  \item $G$ is  {
equationally Noetherian}, i.e.,  every system $S(X) = 1$
 over $G$ is equivalent to some finite part of itself.

  \item the {Zariski topology} (formed by  algebraic sets as a sub-basis of closed sets) over $G^n$ is {\bf Noetherian} for every $n$,
i.e., every proper descending chain of closed sets in $G^n$ is
finite.

 \item Every { chain of proper epimorphisms} of coordinate groups over
 $G$ is {finite}.
 \end{itemize}

If the Zariski topology is Noetherian, then every algebraic set can
be uniquely presented as a  finite union of its {\bf irreducible
components}: $$ V = V_1 \cup \ldots V_k$$

Recall, that a closed subset $V$ is { irreducible} if it is not a
union of two proper closed (in the induced topology) subsets.

\subsection{Fully residually $G$ groups}

A direct limit of a direct system of finite partial $n$-generated
subgroups of $G$ such that all products of generators and their
inverses eventually appear in these partial subgroups, is called a
{\bf limit group over $G$.} The same definition can be given using
the notion of ``marked group''.

 A {\em marked} group $(G,S)$ is a group $G$ with a
prescribed family of generators $S = (s_1,\ldots,s_n)$. Two marked
groups $(G, (s_1,\ldots,s_n))$ and $(G', (s'_1,\ldots,s'_n))$ are
isomorphic as marked groups if the bijection $s_i
\longleftrightarrow s'_i$ extends to an isomorphism. For example,
$(\langle a \rangle,(1,a))$ and $(\langle a \rangle,(a,1))$ are not
isomorphic as marked groups. Denote by ${\cal G}_n$ the set of
groups marked by $n$ elements up to isomorphism of marked groups.
One can define a metric on ${\cal G}_n$ by setting the distance
between two marked groups $(G, S)$ and $(G',S')$ to be $e^{-N}$ if
they have exactly the same relations of length at most $N$.
(This metric was used in \cite{Gromov}, \cite{Grig}, \cite{CG}.)
Finally, a limit group over $G$ is a limit (with respect to the
metric above) of marked groups $(H_i,S_i),$ where $H_i\leq G$, $i\in {\mathbb N},$ in
${\cal G}_n$.

The following two theorems summarize properties that are equivalent
for a group $H$ to the property of being discriminated by $G$ (being
$G$-discriminated by $G$).

\medskip
\noindent {\bf Theorem A} {\it [No coefficients] Let $G$ be an
equationally Noetherian group.
Then for a finitely generated group $H$ the following
conditions are equivalent:
\begin{enumerate}

\item $\Th_{\forall} (G) \subseteq \Th_{\forall} (H)$, i.e., $H \in \ucl(G)$;
\item $\Th_{\exists} (G) \supseteq
    \Th_{\exists} (H)$;
\item $H$ embeds into an ultrapower of $G$;
\item $H$ is discriminated by $G$;
\item $H$ is a limit group over $G$;
\item $H$ is  defined by a complete atomic type in the theory $\Th _{\forall} (G)$;
\item $H$ is the coordinate group of an irreducible
algebraic set over $G$ defined by a system of coefficient-free
equations.
\end{enumerate}
}

For a group $A$ we denote by $\L_A$ the language of groups with constants from $A$.

\medskip
\noindent {\bf Theorem B} {\it [With coefficients] Let $A$ be a group  and $G$ an
$A$-equationally Noetherian $A$-group. Then for a finitely
generated $A$-group $H$ the following conditions are
equivalent:
\begin{enumerate}

\item $\Th_{\forall,A} (G)
    = \Th _{\forall,A} (H)$;
\item $\Th_{\exists,A} (G) = \Th _{\exists,A} (H)$;
\item $H$ $A$-embeds into an  ultrapower of $G$;
\item $H$ is $A$-discriminated by $G$;
\item $H$ is a limit  group    over  $G$;
\item $H$ is a group  defined by a complete atomic type in
    the theory $\Th_{\forall,A} (G)$ in the language $\L_{A}$;
\item $H$ is the  coordinate group  of an irreducible algebraic
set over $G$ defined by  a system of equations with
coefficients in $A$.
\end{enumerate}
}

Equivalences $1\Leftrightarrow 2\Leftrightarrow 3$ are standard
results in mathematical logic. We refer the reader to \cite{R} for
the proof of $2\Leftrightarrow 4$, to \cite{KM9}, \cite{BMR1} for
the proof of $4\Leftrightarrow 7$. Obviously, $2\Rightarrow
5\Rightarrow 3$. The above two theorems are proved in \cite{DMR} for
arbitrary equationally Noetherian algebras. Notice, that in the case
when $G$ is a free group and $H$ is finitely generated, $H$ is a
limit group if and only if it is a limit group in the terminology of
\cite{S1}, \cite{CG} or \cite{Groves1}, \cite{Groves2}.

\subsection{Lyndon's completions of CSA groups}

The paper \cite{MR2}, following Lyndon \cite{Ly2}, introduced a
$\mathbb{Z}[t]$-completion $G^{\mathbb{Z}[t]}$ of a given CSA-group
$G$. In \cite{BMR2} it was shown that if $G$ is a CSA-group
satisfying the Big Powers condition, then finitely
generated subgroups of $G^{\mathbb{Z}[t]}$ are $G$-universally
equivalent to $G$.

We refer to finitely generated $G$-subgroups of $G^{\mathbb{Z}[t]}$
as {\em exponential extensions} of $G$ (they are obtained from $G$
by iteratively adding  $\mathbb{Z}[t]$-powers of  group elements).
The group $G^{\mathbb{Z}[t]}$ is a union of an
ascending chain of extensions of centralizers of the group $G$ (see
\cite{MR2}).

A group obtained as a union of a chain of extensions of
centralizers
 $$\Gamma = \Gamma_0 <  \Gamma _1 < \ldots <  \ldots \cup \Gamma_k $$
  where
   $$\Gamma_{i+1} = \langle \Gamma_i, t_i \mid [C_{\Gamma_i}(u_i),t_i] =  1\rangle$$
(extension of the centralizer $C_{\Gamma_i}(u_i)$) is called an
iterated extension of centralizers and is denoted $\Gamma (U,T),$
where $U=\{u_1,\ldots ,u_k\}$ and $T=\{t_1,\ldots ,t_k\}.$

Every exponential extension $H$ of $G$ is also a subgroup of an
iterated extension of centralizers of $G$.

\subsection{Relatively hyperbolic groups}

A group $G$ is hyperbolic relative to a collection of subgroups $\{H_{\lambda}\}_{\lambda\in\Lambda }$
(parabolic subgroups) if $G$ is finitely presented relative to $\{H_{\lambda}\}_{\lambda\in\Lambda }$
$$G=\langle X\cup (\mathcal H=\bigsqcup _{\lambda\in\Lambda}H_{\lambda})| \mathcal R\rangle ,$$
 and there is a constant $L>0$ such that for any word $W\in (X\cup {\mathcal H})$ representing the identity in $G$ we have ${\rm Area} ^{rel} (W)\leq L||W||,$
where ${\rm Area} ^{rel} (W)$ is the minimal number $k$ such that
$W=\prod _{i=1}^kg_iR_ig_i^{-1},\ r_i\in\mathcal R$, in the free product of the free group with basis $X$ and groups
$\{H_{\lambda}\}_{\lambda\in\Lambda }$.

 In \cite{Groves2}
(Theorem 5.16) Groves showed that groups from $\mathcal{G}$ are
equationally Noetherian.
  By Theorem 1.14 of \cite{RHG}
 the centralizer of every hyperbolic element from a group $G\in\mathcal G$ is cyclic.
 Therefore any non-cyclic abelian subgroup is contained in a finitely generated
 parabolic subgroup. It follows that finitely generated groups from  $\mathcal G$
 are CSA, that is have malnormal maximal abelian subgroups.
(see also Lemma 6.7, \cite{Groves1}).
\subsection{Big Powers condition}
We say that an element $g\in G$ is {\it hyperbolic} if it is not
conjugate to an element of one of the subgroups $H_\lambda $,
$\lambda \in \Lambda $.

\begin{prop} \label{bigpowers} Groups from $\mathcal G$ satisfy the big powers condition for hyperbolic elements: if
$U$ is a set of
 hyperbolic elements,
$g=g_1u_1^{n_1}g_2\ldots u_k^{n_k}g_{k+1},$ $u_1,\ldots , u_k \in
U$,  and $g_{i+1}^{-1}u_ig_{i+1}$ do not commute with $u_{i+1}$, then
there exists a positive number $N$ such that for $|n_i|\geq N, \
i=1,\ldots ,k$,\  $g\neq 1$.\end{prop}

The proof of this proposition is similar to that of \cite[Lemma 4.4]{ESBG} and was
suggested by D. Osin.

The Cayley
graph of $G$ with respect to the generating set $X\cup \mathcal H$
is denoted by $\G$. For a path $p$ in $\G $, $l(p)$ denotes its
length, $p_-$ and $p_+$ denote the origin and the terminus of $p$,
respectively.

\begin{defn}[\cite{RHG}]
Let $q$ be a path in the Cayley graph $\G $. A (non--trivial)
subpath $p$ of $q$ is called an {\it $H_\lambda $--component} for
some $\lambda \in \Lambda $ (or simply a {\it component}), if
\begin{enumerate}
\item[(a)] The label of $p$ is a word in the alphabet
$H_\lambda\setminus \{ 1\} $;

\item[(b)] $p$ is not contained in a bigger subpath of $q$
satisfying (a).
\end{enumerate}
Two $H_\lambda $--components $p_1, p_2$ of a path $q$ in $\G $ are
called {\it connected} if there exists a path $c$ in $\G $ that
connects some vertex of $p_1$ to some vertex of $p_2$ and the label
of the path, denoted ${\phi (c)}$, is a word consisting of letters
from $ H_\lambda\setminus\{ 1\} $. In algebraic terms this means
that all vertices of $p_1$ and $p_2$ belong to the same coset
$gH_\lambda $ for a certain $g\in G$. Note that we can always assume
that $c$ has length at most $1$, as every nontrivial element of
$H_\lambda \setminus\{ 1\} $ is included in the set of generators.
An $H_\lambda $--component $p$ of a path $q$ is called {\it isolated
} (in $q$) if no distinct $H_\lambda $--component of $q$ is
connected to $p$.
\end{defn}

The following lemma can be found in \cite[Lemma 2.7]{Osi}.

\begin{lemma}\label{Omega}
Suppose that $G$ is a group hyperbolic relative to a collection of
subgroups $\Hl $. Then there exists a constant $K>0$ and finite
subset $\Omega \subseteq G$ such that the
following condition holds. Let $q$ be a cycle in $\G $, $p_1,
\ldots , p_k$ a set of isolated components of $q$
for some $\lambda\in \Lambda $, $g_1, \ldots , g_k$ the elements
of $G$ represented by the labels of $p_1, \ldots , p_k$,
respectively. Then for any $i=1, \ldots , k$, $g_i$ belongs to the
subgroup $\langle \Omega \rangle \le G$ and the word
length of $g_i$ with respect to $\Omega $ satisfies the
inequality
$$ \sum\limits_{i=1}^k |g_i|_{\Omega }\le Kl(q).$$
\end{lemma}

Recall also that a subgroup is {\it
elementary} if it contains a cyclic subgroup of finite index. The
lemma below is proved in \cite{ESBG}.

\begin{lemma}\label{Eg}
Let $g$ be a hyperbolic element of infinite order in $G$. Then
\begin{enumerate}
\item The element $g$ is contained in a unique maximal elementary
subgroup $E_G(g)$ of $G$.

\item The group $G$ is hyperbolic relative to the collection
$\Hl\cup \{ E_G(g)\} $.
\end{enumerate}
\end{lemma}

\begin{proof} Now we can prove Proposition \ref{bigpowers}.

It suffices to prove the proposition under the following additional assumption:
if $u_i$ and $u_j$ are conjugate, then $u_i=u_j$, and if $u_i=u_{i+1}$, then $g_{i+1}\not\in E(u_i)$.
Indeed, if $u_j=h^{-1}u_ih$, we replace $u_j$ by $\bar u_j=u_i=hu_jh^{-1}$, $\ g_j$ by $\bar g_j=g_jh^{-1}$
and $g_{j+1}$ by ${\bar g}_{j+1}=hg_{j+1}$. If $[g_j^{-1}u_{j-1}g_j,u_j]\neq 1,$ then
$h[g_j^{-1}u_{j-1}g_j,u_j]h^{-1}= [\bar g_j^{-1}u_{j-1}\bar g_j,\bar u_j]\neq 1.$ Similarly, if
$[g_{j+1}^{-1}u_{j}g_{j+1},u_{j+1}]\neq 1$, then $[\bar g_{j+1}^{-1}\bar u_{j}\bar g_{j+1},u_{j+1}]\neq 1$.
The CSA condition implies that $[g_{i+1}^{-1}u_{i}g_{i+1},u_i]=1$ is equivalent to  $g_{i+1}\in E(u_i).$

Joining $g_1, \ldots , g_{k+1}$ to the finite relative generating set
$X$ if necessary, we may assume that $g_1, \ldots , g_{k+1}\in X$. Set
$${\mathcal F} =\{ f\in \langle \Omega \rangle ,\;
|f|_{\Omega }\le 4K\} ,$$ where $K$ and $\Omega $ are given by Lemma \ref{Omega}. Suppose that $g_1u_1^{n_1}\ldots
g_ku_k^{n_k}g_{k+1}= 1$. We consider a loop $p=q_1r_{1}q_2r_2\ldots q_kr_{k}q_{k+1}$ in
$\G $, where $q_i$ (respectively, $r_{i}$) is labeled by $g_i$
(respectively by $u_i^{n_i}$).

\begin{figure}
\unitlength=1mm \linethickness{0.4pt}
\begin{picture}(140,23)(-25,4)
\qbezier(65.23,7.95)(36.15,35.89)(6.01,7.95)

\put(5.83,7.95){\circle*{1}} \put(15.94,15.68){\circle*{1}}
\put(55.92,15.23){\circle*{1}} \put(50.28,18.65){\circle*{1}}
\put(22.03,18.95){\circle*{1}} \put(65.29,7.8){\circle*{1}}
\put(37.01,21.9){\vector(1,0){.07}}

\put(13,9){\makebox(0,0)[cc]{$p_m$}}
\put(16.5,19.62){\makebox(0,0)[cc]{$r_i$}}
\put(54.5,19.3){\makebox(0,0)[cc]{$r_j$}}
\put(35,11.4){\makebox(0,0)[cc]{$e$}}
\put(35.64,25){\makebox(0,0)[cc]{$s$}}

\thicklines \linethickness{1pt}
\qbezier(15.76,15.61)(18.28,16.95)(22,18.88)
\qbezier(50.09,18.58)(53.74,16.65)(55.89,15.31)
\qbezier(22,18.88)(33.52,9.51)(50.09,18.58)
\put(54.07,16.6){\vector(2,-1){.07}}
\put(20,17.7){\vector(2,1){.07}}
\put(36.7,14.1){\vector(1,0){.07}}

\end{picture}
  \caption{}\label{ahm}
\end{figure}

Note that $r_1, \ldots , r_k$ are components of $p$.
First assume that not all of these components are isolated in $p$.
Suppose that $r_i$ is connected to $r_j$ for some $j>i$ and $j-i$
is minimal possible. Let $s$ denote the segment $[(r_i)_+,
(r_j)_-]$ of $p$, and let $e$ be a path of length at most $1$ in
$\G $ labeled by an element of $H_\lambda $ such that
$e_-=(r_i)_+$, $e_+=(r_j)_-$ (see Fig. \ref{ahm}). If $j=i+1$,
then $Lab (s)=g_{i+1} $. This contradicts the assumption
$g_{i+1}\notin E(u_i)$ since $Lab (s)$ and $Lab (e)$
represent the same element in $G$. Therefore, $j=i+1+l$ for some
$l\ge 1.$ Note that the components $r_{i+1}, \ldots , r_{i+1+l}$ are
isolated in the cycle $se^{-1}$. (Indeed otherwise we can pass to
another pair of connected components with smaller
value of $j-i$.) By Lemma \ref{Omega} we have $u_q^{n_q}\in \langle
\Omega \rangle $ for all $i+1\le q\le i+1+l$ and
$$
\sum\limits_{q=i+1}^{i+l+1} |u_q^{n_q}|_{\Omega }\le
Kl(se^{-1})=K (2k+2).
$$
Hence $|u_p^{n_p}|_{\Omega }\le K (2+2/k)\le 4K$ for at least
one $p$ which is impossible for large $n_p$. Thus all components
$r_1, \ldots , r_k $ are isolated in $p$. Applying now Lemma
\ref{Omega} again, we obtain
$$
\sum\limits_{q=1}^{m} |u_q^{n_q}|_{\Omega }\le Kl(p)=K (2k+2).
$$
This is again impossible for large $n_1,\ldots, n_k$.
\end{proof}

\subsection{Main results and the scheme of the proof}

Our main result is the following theorem.

\medskip
{\bf Theorem C.} [With constants] {\it  Let $\Gamma\in{\mathcal G}$.
A finitely generated $\Gamma$-group $H$ is $\Gamma$-universally
equivalent to  $\Gamma$ if and only if $H$ is embeddable into
$\Gamma ^{\mathbb{Z}[t]}$. }

\smallskip

 The group $\Gamma ^{\mathbb{Z}[t]}$ is discriminated by $\Gamma$. Indeed, it is enough to prove that any group $H$ obtained from $\Gamma$ by a finite series of extensions of centralizers is $\Gamma$-discriminated. We can obtain $H$ from $\Gamma$ in two steps. Let $K$ be a subgroup of $H$ that is obtained from $\Gamma$ by only extending centralizers of elements from parabolic subgroups. Then $K\in {\mathcal G}$ and $H$ is obtained from $K$ by a series of extensions of centralizers of hyperbolic elements. By Proposition \ref{bigpowers} applied to each centralizer extension, $H$ is discriminated by $K$. Since $K$ is discriminated by $\Gamma$ by Lemma \ref{Omega}, $H$ is also discriminated by $\Gamma$.

 The proof of the converse follows the argument in \cite{KMNull},
\cite{KMIrc} with
 necessary modifications. It splits into steps. In Section 3 we will
 prove

{\bf Theorem D.}  {\it Let $\Gamma\in\mathcal G$
 and $H$ a finitely generated group discriminated by $\Gamma$. Then $H$
embeds into an NTQ extension of $\Gamma$.}

\smallskip

\smallskip
In Section 4 we will prove

 {\bf Theorem E.}  {\it Let $\Gamma\in
\mathcal G$ and $\Gamma^\ast$ an NTQ extension of $\Gamma$.   Then
$\Gamma ^\ast$ embeds into a group $\Gamma (U,T)$ obtained from
$\Gamma$ by finitely many extensions of centarlizers.}

\section{Quadratic equations and NTQ systems and groups}
\begin{definition}
A standard quadratic equation over the group $G$ is an equation of
the one of the following forms (below $d,c_i$ are nontrivial
elements from $G$):
\begin{equation}\label{eq:st1}
\prod_{i=1}^{n}[x_i,y_i] = 1, \ \ \ n > 0;
\end{equation}
\begin{equation}\label{eq:st2}
\prod_{i=1}^{n}[x_i,y_i] \prod_{i=1}^{m}z_i^{-1}c_iz_i d = 1,\ \ \
n,m \geq 0, m+n \geq 1 ;
\end{equation}
\begin{equation}\label{eq:st3}
\prod_{i=1}^{n}x_i^2 = 1, \ \ \ n > 0;
\end{equation}
\begin{equation}\label{eq:st4}
\prod_{i=1}^{n}x_i^2 \prod_{i=1}^{m}z_i^{-1}c_iz_i d = 1, \ \ \ n,m
\geq 0, n+m \geq 1.
\end{equation}

Equations (\ref{eq:st1}), (\ref{eq:st2}) are called {\em orientable}
of genus $n$, equations (\ref{eq:st3}), (\ref{eq:st4}) are called
{\em non-orientable} of genus $n$.
\end{definition}

Let $W$ be a strictly quadratic word over a group $G$. Then there is
a $G$-automorphism $f \in Aut_G(G[X])$ such that  $W^f$ is a
standard quadratic word over $G.$

To each quadratic equation one can associate a punctured surface.
For example, the orientable surface associated to equation
\ref{eq:st2} will have genus $n$ and $m+1$ punctures.

\begin{definition}
Strictly quadratic words  of the type $ [x,y], \ x^2, \ z^{-1}cz,$
where $c \in G$, are called {\em atomic quadratic words} or simply
{\em atoms}.
\end {definition}

By definition a standard quadratic equation $S = 1$  over $G$  has
the form $$ r_1 \ r_2 \ldots r_{k}d  = 1,$$ where $r_i$ are atoms,
$d \in G$.  This number $k$ is called the {\it atomic rank } of this
equation,  we denote it by $r(S)$.

\begin{definition}  Let $S = 1$ be a standard quadratic equation written in the
atomic form
 $r_1r_2\ldots r_kd = 1 $ with $k \geq 2$.  A solution $\phi : G_{R(S)}
\rightarrow G$
 of $S = 1$  is called:
 \begin{enumerate}
 \item degenerate, if $r_i^\phi = 1$ for some $i$, and
 non-degenerate otherwise;
 \item  commutative, if $[r_i^{\phi},r_{i+1}^{\phi}]=1$ for all
$i=1,\ldots ,k- 1,$  and  non-commutative otherwise;
 \item in a general position, if $[r_i^{\phi},r_{i+1}^{\phi}] \neq 1$ for all
$i=1,\ldots ,k-1,$.
 \end{enumerate}
 \end{definition}

 Put
 $$\kappa(S) = |X| + \varepsilon(S),$$
  where $\varepsilon(S) = 1$  if $S$ of the type (\ref{eq:st2}) or
  (\ref{eq:st4}), and  $\varepsilon(S) = 0$ otherwise.

\begin{definition}\label{regular}
Let  $S=1$ be a standard quadratic equation over a group $G$ which
has a solution in $G$. The equation $S(X) = 1$ is  {\em regular} if
$\kappa(S) \geq 4$ (equivalently, the Euler characteristic of the
corresponding punctured surface is at most -2) and there is a
non-commutative solution of $S(X) = 1$ in $G$, or it is an equation
of the type $[x,y]d = 1$ or $[x_1,y_1][x_2,y_2]=1$.
\end{definition}

Let $G$ be a group with a generating set $A$. A system of equations
$S = 1$  is called {\em triangular quasi-quadratic} (shortly, TQ)
over $G$ if it can be partitioned into the following subsystems

\medskip
$S_1(X_1, X_2, \ldots, X_n,A) = 1,$

\medskip
$\ \ \ \ \ S_2(X_2, \ldots, X_n,A) = 1,$

$\ \ \ \ \ \ \ \ \ \  \ldots$

\medskip
$\ \ \ \ \ \ \ \ \ \ \ \ \ \ \ \ S_n(X_n,A) = 1$

\medskip \noindent
 where for each
$i$ one of the following holds:
\begin{enumerate}
\item [1)] $S_i$ is quadratic  in variables $X_i$;
 \item [2)] $S_i= \{[y,z]=1, [y,u]=1 \mid y, z \in X_i\}$ where $u$ is a
group word in $X_{i+1} \cup  \ldots \cup X_n \cup A$. In this case
we say that $S_i=1$ corresponds to an extension of a centralizer;
 \item [3)] $S_i= \{[y,z]=1 \mid y, z \in X_i\}$;
 \item [4)] $S_i$ is the empty equation.
  \end{enumerate}

Sometimes, we join several consecutive subsystems $S_i = 1, S_{i+1}
= 1, \ldots, S_{i+j}= 1$  of a TQ system $S = 1$ into one block,
thus partitioning   the system $S = 1$ into new blocks. It is
convenient to call a new system also a triangular quasi-quadratic
system.

In the notations above define $G_{i}=G_{R(S_{i}, \ldots, S_n)}$ for
$i = 1, \ldots, n$ and put $G_{n+1}=G.$ The  TQ system $S = 1$ is
called {\em non-degenerate} (shortly, NTQ) if the following
conditions hold:
 \begin{enumerate}
  \item [5)]  each system $S_i=1$, where $X_{i+1}, \ldots, X_n$
are viewed as the corresponding constants from $G_{i+1}$ (under the
canonical maps $X_j \rightarrow G_{i+1}$, $j = i+1, \ldots, n$)  has
a solution in $G_{i+1}$;
 \item [6)] the element  in $G_{i+1}$ represented by the word $u$ from 2) is  not
   a proper power in $G_{i+1}$.
  \end{enumerate}

 An NTQ system $S = 1$ is called {\em regular} if each non-empty quadratic
equation in $S_i$ is regular (see Definition \ref{regular}). The
coordinate group of an NTQ system (regular NTQ system) is called an
{\em NTQ group} (resp., {\em regular NTQ group}).

\section{Embeddings into NTQ extensions}

Let $\Gamma\in\mathcal G$. In this section we will prove Theorem D.
Namely, we will show how to embed a finitely generated fully
residually $\Gamma$ group into an NTQ extension of $\Gamma$.

\begin{theorem}[Theorem 1.1, \cite{Groves2}] Let $\Gamma\in\mathcal G$ and $G$ a
finitely generated freely indecomposable group with abelian JSJ
decomposition $\mathcal D$. Then there exists a finite collection $\{\eta
_i:G\rightarrow L_i\}_{i=1}^n$ of proper quotients of $G$ such that,
for any homomorphism $h:G\rightarrow \Gamma$ which is not equivalent
to an injective homomorphism there exists $h^{\prime}:G\rightarrow
\Gamma$ with $h\sim h'$ (the relation $\sim $ uses conjugation,
canonical automorphisms corresponding to $\mathcal D$ and "bending moves" ),
$i\in\{1,\ldots,n\}$ and $h_i:L_i\rightarrow \Gamma$ so that
$h'=\eta _ih_i.$ The quotient groups $L_i$ are fully residually
$\Gamma$.\end{theorem}

This theorem reduces the description of $Hom (G,\Gamma)$ to a
description of $Hom (L_i,\Gamma)_{i=1}^n$. We then apply it again to
each $L_i$ in turn and so on with successive proper quotients. Such
a sequence terminates by equationally Noetherian property.  Using
this theorem one can construct a $Hom$-diagram which is the same as
a so-called Makanin-Razborov constructed in Section 6 of
\cite{Groves2}.

The statement of the above theorem is still true if we replace the
set of all homomorphisms $h:G\rightarrow \Gamma$ by the set of all
$\Gamma$-homomorphisms. The proof is the same. Therefore, a similar
diagram can be constructed for $\Gamma$-homomorphisms $G\rightarrow
\Gamma$.

\medskip

{\em Proof of Theorem D.}  Let $G$ be a finitely generated freely indecomposable group discriminated by $\Gamma$. According to the construction of
Makanin-Razborov diagram the set $Hom (G,\Gamma )$ is divided into a
finite number of families. Therefore one of these families contains
a discriminating set of homomorphisms. Each family corresponds to a
sequence of fully residually $\Gamma$ groups (see \cite{KMel})
$$ G=G_0,G_1,\ldots ,G_n,$$ where $G_{i+1}$ is a proper quotient of
$G_i$ and $\pi _{i}:G_i\rightarrow G_{i+1}$ is an epimorphism.
Similarly to Lemma 16 from \cite{KMel}, for a discriminating family
$\pi _i$ is a monomorphism for the following subgroups $H$ in the
JSJ decomposition ${\mathcal D}_i$ of $G_i$
\begin{enumerate}
 \item  $H$ is  a rigid subgroup in ${\mathcal D}_i$;
  \item $H$ is  an edge subgroup in  ${\mathcal D}_i$;
 \item $H$ is  the subgroup of an abelian
vertex groups $A$ in ${\mathcal D}_i$  generated by the canonical images in $A$
of the edge groups of the edges of ${\mathcal D}_i$ adjacent to $A$.
\end{enumerate}

We need the following result.
 \begin{lemma}[Lemma 22, \cite{KMel}]
 \label{le:extend-hom-amalgam}\
 \begin{enumerate}
 \item[(1)] Let $H = A\ast_D B$, $D$ be abelian subgroup that is maximal abelian in $A$ or $B$, and $\pi:H \rightarrow {\bar
 H}$ be a homomorphism such that the restrictions of $\pi$ on $A$
 and $B$ are injective. Put
  $$H^\ast = \langle {\bar H}, y \mid [C_{{\bar H}}(\pi(D)),y] = 1
  \rangle.$$
   Then for every $u \in C_{H^\ast}((\pi(D))$,  $u \not \in C_{{\bar
   H}}(\pi(D))$, a  map
    $$\psi(x) = \left \{ \begin{array}{lr}  \pi(x), & \  x \in A,
    \\
                                    \pi(x)^u, & \ x \in B.
                         \end{array}
                         \right. $$
    gives rise to a monomorphism $\psi:H \rightarrow H^\ast$.

    \item[(2)] Let $H = \langle A, t \mid d^t = c, d\in D\rangle ,$ where $D$ is abelian and either $D$ or its image is maximal abelian in $A$,  and $\pi:H \rightarrow {\bar
 H}$ be a homomorphism such that the restriction of $\pi$ on $A$
 is injective. Put
  $$H^\ast = \langle {\bar H}, y \mid [C_{{\bar H}}(\pi(D)),y] = 1
  \rangle.$$
   Then for every $u \in C_{H^\ast}((\pi(D))$,  $u \not \in C_{{\bar
   H}}(\pi(D))$, a map
    $$\psi(x) = \left \{ \begin{array}{lr}  \pi(x), & \  x \in A,
    \\
                                    u\pi(x), & \ x = t.
                         \end{array}
                         \right. $$
    gives rise to a monomorphism $\psi:H \rightarrow H^\ast$.
    \end{enumerate}
\end{lemma}

Let now ${\mathcal D}$ be an abelian JSJ decomposition of $G$. Combining foldings and slidings, we can transform ${\mathcal D}$
into an abelian decomposition in which each vertex with non-cyclic
abelian subgroup that is connected to some rigid vertex, is
connected to only one vertex which is rigid. We suppose from the
beginning that ${\mathcal D}$ has this property. Let $G_1$ be the fully residually $\Gamma$ proper quotient of $G$ on the next level of the Makanin-Razborov diagram, and $\pi$ be the canonical epimorphism $\pi :G\rightarrow G_1$.  Let $G_1=P_1\ast\cdots\ast P_{\alpha}\ast F,$  be the Grushko
decomposition of $G_1$ relative to the set of all rigid subgroups and edge subgroups of ${\mathcal D}$.
Here $F$ is the free factor and each $P_i$ is freely indecomposable modulo rigid subgroups and edge subgroups of ${\mathcal D}.$

We will
construct a canonical extension $G^\ast$ of $\bar G=P_1\ast\cdots\ast P_{\alpha}$ which is a
fundamental group of the graph of groups $\Lambda$ obtained from a
single vertex $v$ with the associated vertex group $G_v=\bar G$ by
adding finitely many edges corresponding to extensions of
centralizers (viewed as amalgamated products) and finitely many
QH-vertices connected only to $v$.
By construction of $\bar G$, each
factor in this decomposition contains a conjugate of the image of
some rigid subgroup or an edge group in ${\mathcal D}$.
Indeed, the Grushko decomposition of $\bar G$ is non-trivial only if the
fundamental groups of some separating simple closed curves on the surfaces corresponding to $QH$ subgroups of $\mathcal D$ are mapped by $\pi$ to the identity element. Such curves cut the surface into pieces, and the fundamental groups of all the pieces that are not attached to rigid subgroups are mapped into $F$.

Let
$g_1,\ldots,g_l$ be a fixed finite generating set of ${\bar G}$. For
an edge $e \in {\mathcal D}$ we fix a tuple of generators  $d_e$ of
the abelian edge group $G_e$.
 The required extension $G^\ast$ of ${\bar G}$ is
constructed  in three steps. On each step we extend the centralizers
$C_{\bar G}(\pi(d_e))$ of some edges $e$ in ${\mathcal D}$ or add a
QH subgroup. Simultaneously, for every edge $e \in {\mathcal D}$ we
associate an element $s_e \in C_{G^\ast}(\pi(d_e))$.

{\em Step 1.} Let $E_{rig}$ be the set of all edges between rigid
subgroups in $\mathcal D$.  One can define an equivalence relation
$\sim$ on $E_{rig}$  assuming for $e, f \in E_{rig}$ that
 $$e \sim f \Longleftrightarrow \exists g_{ef} \in {\bar G}
 \left ( g_{ef}^{-1}C_{\bar G}(\pi(e))g_{ef} = C_{\bar G}(\pi(f)) \right).
  $$
Let $E$ be a set of representatives of equivalence classes of
$E_{rig}$ modulo $\sim$. Now we construct a group $G^{(1)}$  by
extending every centralizer $C_{\bar G}(\pi(d_e))$ of ${\bar G}$, $e
\in E$ as follows. Let
$$[e] = \{e = e_1, \ldots, e_{q_e}\}$$
and $y_e^{(1)}, \ldots, y_e^{(q_e)}$  be new letters corresponding
to the elements in $[e]$. Then put
 $$G^{(1)} = \langle {\bar G}, y_e^{(1)},
\ldots, y_e^{(q_e)} (e \in E) \mid [C(\pi(d_e)),y_e^{(j)}]
 = 1, [y_e^{(i)},y_e^{(j)}] = 1 (i,j = 1, \ldots, q_e)\rangle.$$

One can  associate with $G^{(1)}$  the following system of equations
over $\bar G$:

\begin{equation}\label{eq:edges}
[{\bar g_{es}},y_e^{(j)}]=1,\ [y_e^{(i)},y_e^{(j)}] = 1, \ \ i,j= 1,
\ldots, q_e, \ s = 1, \ldots, p_e, \ e\in E,
\end{equation}
 where $y_e^{(j)}$ are new variables and the elements ${\bar g_{e1}}, \ldots, {\bar g_{ep_e}}$
 are constants from ${\bar G}$  which generate  the centralizer $C(\pi(d_e))$.
 We assume that the constants ${\bar g_{ej}}$
 are given as words in the generators $g_1, \ldots, g_l$ of ${\bar
 G}$. We associate with the edge $e_i\in [e]$ an element $s_{e_i}$ that is the conjugate of $y_e^{(i)}$ from $C_{G^{(1)}}(\pi(d_{e_i}))$.

{\em  Step 2.} Let $A$ be a non-cyclic abelian vertex group in
$\mathcal D$ and $A_e$ the subgroup of $A$ generated by the images
in $A$ of the edge groups of edges adjacent to $A$. Then $A =
Is(A_e) \times A_0$ where $Is(A_e)$ is the isolator of $A_e$ in $A$
(the minimal direct factor containing $A_e$) and $A_0$ a  direct
complement of $Is(A_e)$ in $A$. Notice, that the restriction of
$\pi_1$ on $Is(A_e)$ is a monomorphism (since $\pi _1$ is injective
on $A_e$ and $A_e$ is of finite index in $Is(A_e)$). For each
non-cyclic abelian vertex group $A$ in $\mathcal D$ we extend the
centralizer of $\pi _1(Is(A_e))$ in $G^{(1)}$ by the abelian group
$A_0$ and denote the resulting group by $G^{(2)}$. Observe, that
since $\pi _1(Is(A_e)) \leq {\bar G}$ the group $G^{(2)}$ is
obtained from ${\bar G}$ by extending finitely many centralizers of
elements from ${\bar G}$.

 If the  abelian group $A_0$ has rank
$r$ then the system of equations associated with the abelian vertex
group $A$ has the following form
\begin{equation}\label{eq:barS} [y_p, y_q]=1, [y_p, {\bar d_{ej}}]=1, \ \
 \  p,q=1,\ldots ,r, j = 1, \ldots, p_e,\end{equation}
 where $y_p, y_q$ are new variables and the elements ${\bar d_{e1}}, \ldots, {\bar
d_{ep_e}}$
 are constants from ${\bar G}$  which generate  the subgroup $\pi(Is(A_e))$.
 We assume that the constants ${\bar d_{ej}}$
 are given as words in the generators $g_1, \ldots, g_l$ of ${\bar
 G}$.

{\em Step 3.} Let $Q$ be a non-stable QH subgroup in $\mathcal D$ (not mapped by $\pi$ into the same QH subgroup).
Suppose $Q$ is given by a presentation
$$\prod _{i=1}^{n}[x_i, y_i]p_1\cdots p_{m}=1.$$
 where  there are exactly $m$
outgoing edges $e_1,\ldots ,e_m$ from $Q$  and $\sigma(G_{e_i})=
\langle p_i\rangle $,  $\tau(G_{e_i}) = \langle c_i\rangle $ for
each edge $e_i$. We add a QH vertex $Q$ to $G^{(2)}$ by introducing
new generators and the following quadratic relation

\begin{equation}\label{QQQC}\prod _{i=1}^{n}[ x_i,
y_i](c_1^{\pi_1})^{z_1}\cdots
(c_{m-1}^{\pi_1})^{z_{m-1}}c_m^{\pi_1}=1\end{equation}
 to the presentation of $G^{(2)}$.
Observe, that in the relations (\ref{QQQC}) the coefficients in the
original quadratic relations for $Q$ in $\mathcal D$ are
 replaced by their images in $\bar G$.

Similarly,  one introduces QH vertices for non-orientable QH
subgroups in $\mathcal D$.

The resulting group is denoted by $G^\ast = G^{(3)}$.

We  define a ($\Gamma$)-homomorphism $\psi:G \rightarrow G^\ast $
with respect to the splitting $\mathcal D$ of $G$ and will prove
that it is a monomorphism. Let $T$ be the maximal subtree of
$\mathcal D$. First, we define $\psi$ on the fundamental group of
the graph of groups induced from $\mathcal D$ on $T$.
 Notice that if we consider only $\Gamma$-homomorphisms, then the subgroup  $\Gamma$ is elliptic in $\mathcal D$, so there is a
rigid vertex $v_0 \in T$ such that $\Gamma \leq G_{v_0}$. Mapping
$\pi$ embeds $G_{v_0}$ into ${\bar G}$, hence into $G^\ast$.

 Let $P$ be a path $v_0 \rightarrow v_1 \rightarrow \ldots
\rightarrow v_n$ in $T$ that starts at $v_0$. With each edge $e_i =
(v_{i-1} \rightarrow v_i)$ between two rigid vertex groups  we have
already associated the  element $s_{e_i}$. Let us associate elements
to other edges of $P$:

a) if $v_{i-1}$ is a rigid vertex, and $v_i$ is either abelian or
QH, then $s_{e_i}=1$;

b) if $v_{i-1}$ is a QH vertex, $v_i$ is rigid or abelian, and the
image of $e_i$ in the decomposition  ${\mathcal D}^*$ of $G^*$ does
not belong to $T^*$, then $s_{e_i}$ is the stable letter
corresponding to the image of $e_i$;

c) if $v_{i-1}$ is a QH vertex  and $v_i$ is rigid or abelian, and
the image of $e_i$ in the decomposition of $G^*$ belongs to $T^*$,
then $s_{e_i}=1$.

d) if $v_{i-1}$ is an abelian vertex with $G_{v_{i-1}}=A$   and
$v_i$ is a QH vertex, then $s_{e_i}$ is an element from $A$ that
belongs to $A_0$.

Since two abelian vertices  cannot be connected by an edge in
$\Gamma$, and we can suppose that  two QH vertices are not connected
by an edge, these are all possible cases.

We now define the embedding $\psi$ on the fundamental group
corresponding to the path $P$ as follows:
 $$\psi(x) = \pi(x)^{s_{e_i}\ldots s_{e_1}}\ \verb for \ x\in G_{v_i}.$$
This map is a monomorphism by Lemma \ref{le:extend-hom-amalgam}.
Similarly  we define $\psi$ on the fundamental group of the graph of
groups induced from $\mathcal D$ on $T$. We  extend it to $G$ using
the second statement of  Lemma \ref{le:extend-hom-amalgam}.

Recursively applying this procedure to $G_1$ and so on, we will
construct the NTQ group $N$ such that $G$ is embedded into $N$.
Theorem D is proved.

\section {Embedding of NTQ groups into
$G(U,T).$}

An NTQ group $H$ over $\Gamma$ is obtained from $\Gamma$ by a series
of extensions:
$$ \Gamma =H_0<H_1<\ldots H_n=H,$$ where for each $i=1,\ldots , n$, $H_i$ is either an extension of a centralizer in $H_{i-1}$ or the coordinate group of a regular quadratic equation
over $H_{i-1}$. In the second case, equivalently,  $H_i$ is the
fundamental group of the graph of groups with two vertices, $v$ and
$w$ such that $v$ is a QH vertex with QH subgroup $Q$, and $H_{i-1}$
is the vertex group of the second vertex $w$. Moreover,  there is a
retract from $H_i$ onto $H_{i-1}$. In this section we will prove the
following theorem which, by induction, implies Theorem E.
\begin{theorem} \label{quadr} Let $H$ be the fundamental group of the graph of
groups with two vertices, $v$ and $w$ such that $v$ is a QH vertex
with QH subgroup $Q$,  $H_w=\Gamma \in{\mathcal G},$ and there is a
retract from $H$ onto $\Gamma$ such that $Q$ corresponds to a
regular quadratic equation. Then $H$ can be embedded into a group
obtained from $\Gamma$ by a series of extensions of
centralizers.\end{theorem}

The idea of the proof of this theorem is as follows. Let $S_Q$ be a
punctured surface corresponding to the QH vertex group in this
decomposition (denote the decomposition by $\mathcal D$) of $H$. We
will find in Proposition \ref{propmain} a finite collection of
simple closed curves on $S_{Q}$ and a homomorphism $\delta
:H\rightarrow K,$ where $K$ is an iterated centralizer extension of
$\Gamma *F$,
 with the following properties:

1) $\delta$ is a retraction on $\Gamma$,

2) each of the simple closed curves in the collection and all
boundary elements of $S_Q$ are mapped by $\delta$ into non-trivial
elements of $K$,

3) each connected component of the surface obtained by cutting
$S_{Q}$ along this family of s.c.c. has Euler characteristic -1,

4) the fundamental group of each of these connected components is
mapped monomorphically into a 2-generated free subgroup of $K$.

 Given this collection of
s.c.c. on the surface associated with the QH-vertex group in the
decomposition $\mathcal D$, one can extend $\mathcal D$ by further
splitting the QH-vertex groups along the family of simple closed
curves described above. Now the statement of Theorem \ref{quadr}
would follow from Lemma \ref{le:extend-hom-amalgam}.

\begin{prop} [\cite{KMNull}, Prop.3]
\label{prop1} Let $S=1$ be a nondegenerate standard quadratic
equation over a CSA-group $G.$  Then either $S = 1$ has a solution
in general position, or every nondegenerate solution of $S = 1$ is
commutative.
\end{prop}

 Proving the theorem we will consider the following three cases
for the equation corresponding to the QH subgroup $Q$: orientable of
genus $\geq 1,$ genus = 0, and non-orientable of genus $\geq 1.$ For
an orientable equation of genus $\geq 1$ we have the following
proposition.

\begin{prop}\label{prop5}(Compare [\cite{KMNull}, Prop.4])
Let $S: \prod_{i = 1}^{i = m}[x_i,y_i] \prod_{j = 1}^{j =
n}c_j^{z_j}g^{-1}=1$ ($m \geq 1, n \geq 0$) be a nondegenerate
standard quadratic equation over a  group $G\in\mathcal G.$  Then $S
= 1$ has a solution in general position in some group $H$ which is
an iterated extension of centralizers of $G*F$ (where $F$ is a free
group) unless $S = 1$ is the equation $[x_1,y_1][x_2,y_2] = 1$ or
$[x,y]c^z = 1.$ This solution can be chosen so that the images of
$x_i$ and $ y_i$ generate a free subgroup (for each $i=1,\ldots m$).
\end{prop}
{\em Proof of Proposition \ref{prop5}.} Let $n = 0$. In this event
we have a standard quadratic equation of the type
$$
[x_1,y_1] \ldots [x_k,y_k] = g,
$$
which we will sometimes write as $r_1 \ldots r_k = g$, where, as
before, $r_i = [x_i,y_i].$
\begin{lemma}\label{3.4}
Let $S:[x_1,y_1][x_2,y_2] = g$ be a nondegenerate  equation over a
group $G\in \mathcal G.$ Then $S = g$ has a solution in general
position in some group $H$ which is an iterated extension of
centralizers of $G\ast F$ unless $S = 1$ is the equation
$[x_1,y_1][x_2,y_2] = 1.$ Moreover, for each $i$, $x_i, y_i$
generate a free subgroup.
\end{lemma}
\begin{proof} Suppose $S = g$ has   a solution $\phi$ such that $r_1^\phi =
1$ and $r_2^\phi = 1.$ Then $g = 1$ and our equation takes the form
\begin{equation} \label{g2} [x,y][x_2,y_2] = 1.
\end{equation}
 From now on we  assume that for
all solutions $\phi$ either $r_1^\phi \neq 1$ or $r_2^\phi \neq 1.$

Suppose now that just one of the equalities  $r_i^\phi = 1$ ($i =
1,2$) takes place, say $r_1^\phi = 1.$ Write $x_2^\phi = a$, and
$y_2^\phi = b.$ Then the equation is in  the form
$$
[x,y][x_2,y_2] = [a,b] \neq 1.
$$
This equation  has other solutions, for example, for a new letter
$c$ and $p>2$,
\begin{equation} \label{specific} \psi : \ \ x \rightarrow (ca^{-1})^{-p}c, \
y \rightarrow c^{(ca^{-1})^p}, \ x_2 \rightarrow a^{(ca^{-1})^p}, \
y_2 \rightarrow (ca^{-1})^{-p}b \end{equation} for which
$$
r_1^\psi = [c,(ca^{-1})^p] \neq 1 \ \  and \ \  r_2^\psi =
[(ca^{-1})^p,a][a,b] \neq 1.
$$
We claim, that we have $[r_1^\psi,r_2^\psi]
 \neq 1.$ Indeed, $[r_1^\psi,r_2^\psi] = 1$ if and only if $[[c,(ca^{-1})^p], [(ca^{-1})^p,a][a,b]] = 1$,
but this is not true in $G\ast \langle c\rangle.$

Thus, just one case is left to consider. Suppose that
$[r_1^\phi,r_2^\phi] = 1$ and  $r_i^\phi \neq 1$  ($i = 1,2$) for
all solutions $\phi.$  Suppose $x^{\phi}=a,y^{\phi}=b, x_2^{\phi}=c$
and $y_2^{\phi}=d.$ We will use ideas from \cite{Imp} to change the
solution. Let $$H=\langle G,t_1,t_2,t_3,t_4,t_5|
1=[t_1,b]=[t_2,t_1a]=[t_3,d]=[t_4,t_3c]=[t_5,t_2bc^{-1}t_3^{-1}]\rangle
.$$

Let $x^{\psi}=t_5^{-1}t_1a,\ y^{\psi}=(t_2b)^{t_5},\ x_2=(t_3c)^{t_5},\ y_2^{\psi}=t_5^{-1}t_4d.$

This $\psi$ is also a solution of the same equation. But now
$x^{\psi}$ and $y^{\psi}$ generate a free subgroup of $H$. If we
have a word $w(x,y)$ then $w(x^{\psi}, y^{\psi})=1$ in $H$ if all
occurrences of $t_5$ disappear. This can only happen if $w(x,y)$ is
made from the blocks $x^{-1}yx$. But these blocks commute, hence
$w=x^{-1}y^nx$. But now $w^{\psi}=a^{-1}t_1^{-1}(t_2b)^nt_1a,$
therefore $w^{\psi}$ contains $t_2$ that does not disappear.
Therefore $w^{\psi}\neq 1$. Similarly, $x_2^{\psi}$ and $y_2^{\psi}$
generate a free subgroup of $H$.

We will show now that $[r_1^\psi,r_2^\psi]
\not = 1$ . Indeed, $$r_1^\psi r_2^\psi =[x^{\psi},y^{\psi}][x_2^{\psi},y_2^{\psi}]=[a,b][c,d],$$ but
$$r_2^\psi r_1^\psi =[x_2^{\psi},y_2^{\psi}][x^{\psi},y^{\psi}]=t_5^{-1}c^{-1}t_3^{-1}t_5d^{-1}t_3cda^{-1}t_1^{-1}b^{-1}t_2^{-1}t_1at_5^{-1}t_2bt_5.$$
And there is no way to make a pinch and cancel $t_5$ in the second expression.
Therefore $[r_1^\psi,r_2^\psi]
\not = 1$ and the proposition is proved.
\end{proof} Similarly,  one can
prove the following lemma.
\begin{lemma} \label{lemma3.5}(compare \cite{KMNull}, Lemma 13])
Let  $S: [x_1,y_1] \ldots [x_k,y_k] = g$ be a nondegenerate equation
over group $G\in\mathcal G$ and assume that
 $k \geq 3.$ Then $S = g$ has a solution in general position over some group $H$ which is an iterated extension of centralizers of $G\ast F$.
 Moreover, for each $i$, $x_i, y_i$ generate a free subgroup.
\end{lemma}
\begin{proof} The proof will follow by induction on $k$.

Let $k=3$. Assume that $g=1.$ This means we have the equation
$$[x_1,y_1][x_2,y_2][x_3,y_3]=1,$$ which has a solution
$$ x_1^{\phi}=a,\ y_1^{\phi}=b,\ x_2^{\phi}=b,\ y_2^{\phi}=a,\ x_3^{\phi}=1,\ y_3^{\phi}=1,$$
where $a,b$ are arbitrary generators  of $F$. Then the lemma follows
from Proposition 4 \cite{KMNull}. But for convenience of the reader
we will give a proof here. The equation $$[x_2,y_2][x_3,y_3]=[b,a]$$
is nondegenerate of atomic rank 2; hence, by the lemma above, it has
a solution $\theta$ such that $[r_2^{\theta},r_3^{\theta}]\neq 1$,
and the images $x_2^{\theta}, y_2^{\theta}$ (the images
$x_3^{\theta}, y_3^{\theta}$) generate a free non-abelian subgroup.
We got a solution $\psi$, such that
$$x_1^{\psi}=a, y_1^{\psi}=b, x_i^{\psi}=x_i^{\theta},  y_i^{\psi}=y_i^{\theta}, {\rm for}\ i=2,3.$$ Now we are in a position to apply the previous lemma to the equation
$$[x_1,y_1][x_2,y_2]=[y_3^{\psi},x_3^{\psi}].$$
It follows that there exists a solution to $S=g$ in general position
and such that the subgroups generated by the images of $x_i,y_i$ are
free non-abelian for $i=1,2,3$.

Assume now that $g\neq 1$. Then there exists a solution $\phi$ such that for at least one $i$ we have $r_i^{\phi}\neq 1$.
Renaming variables one can assume that exactly $r_3^{\phi}=[a,b]\neq 1, \ a,b\in G
$. Then the equation
$$r_1r_2=g[b,a]$$ has a solution in $G$. Again, we have two cases. If $g[b,a]\neq 1$, then we can argue as in Lemma \ref{3.4}.
We obtain first a solution $\phi$ such that $x_i^{\phi}=c_i,
y_i^{\phi}=d_i, i=1,2,$ $x_3^{\phi}=a, y_3^{\phi}=b$,
$[r_1^{\phi},r_2^{\phi}]\neq 1,$ $[c_1,d_1]\neq g,$ and $c_i,d_i$
generate a free subgroup for $i=1,2.$. Then we consider the equation
$[x_2,y_2][x_3,y_3]=[d_1,c_1]g$ and apply Lemma \ref{3.4} once more.

 If $g[b,a]=1$ then $g=[a,b]$ and the initial equation $S=g$ actually has the form $$r_1r_2r_3=[a,b].$$
In this event consider a solution $\theta$ such that
$$x_1^{\theta}=c, y_1^{\theta}=d, x_2^{\theta}=(ca^{-1})^{-1}d,
y_2^{\theta}=c^{(ca^{-1})}, x_3^{\theta}=a^{(ca^{-1})},
y_3^{\theta}=(ca^{-1})^{-1}b,$$ where $c,d$ are non-commuting
elements from $F$. Then $[r_i^{\theta},r_j^{\theta}]\neq 1,
i,j=1,2,3$, and, obviously, $x_i^{\theta},y_i^{\theta}$ generate a
free group.

Let $k>3$. The equation $$r_1\ldots r_k=g$$
has a solution $\phi$ such that at least for one $i$, say $i=k$ (by renaming variables we can always assume this),
we have $r_k^{\phi}=[a,b]\neq 1.$ Then the equation $$r_1\ldots r_{k-1}=g[b,a]$$ is nondegenerate and by induction
there is a solution $\theta$ such that $[r_i^{\theta},r_{i+1}^{\theta}]\neq 1$ for all $i=1,\ldots ,k-2,$ and $x_i,y_i$ generate a free subgroup for $i=1,\ldots ,k-1.$ Define now a solution $\theta _1$ of the initial equation $S=g$ as follows $$
x_i^{\theta}=x_i^{\theta _1}, y_i^{\theta}=y_i^{\theta _1}, {\rm for} \ i=1,\ldots ,k-2,$$

$$x_{k-1}^{\theta _1}=t_5^{-1}t_1x_{k-1}^{\theta}, y_{k-1}^{\theta _1}=(t_2y_{k-1}^{\theta})^{t_5}, x_k^{\theta _1}=(t_3a)^{t_5}, y_k^{\theta _1}=t_5^{-1}t_4b,$$ where $$[t_1,y_{k-1}^{\theta}]=[t_2, t_1x_{k-1}^{\theta}]=
[t_3,b]=[t_4,t_3a]=[t_5,t_2y_{k-1}^{\theta}a^{-1}t_3^{-1}]=1.$$ This
solution satisfies the requirements of the lemma.

\end{proof}

Thus, Proposition \ref{prop5} is proved for the case $n=0$. Consider
now the case $n>0$.
\begin{lemma} (compare \cite{KMNull}, Lemma 14])
\label{c2} The equation $S: [x,y]c^z = g,$ where $g \neq 1$ which is consistent over a  group
$G\in\mathcal G$ always has a solution in general position in some iterated centralizer extension $H$ of $G$
such that the images of $x$ and $y$ generate a free subgroup.

\end{lemma}

\begin{proof} Let $x \rightarrow a, \ y \rightarrow b, \ z \rightarrow d$ be
an arbitrary solution of $[x,y]c^z = g,$ where $g \neq 1.$ Then $g =
[a,b]c^d$ and the equation takes the form
$$
[x,y]c^z = [a,b]c^d.
$$
We can assume that $[a,b] \neq 1.$ Indeed, suppose $[a,b] = 1$. If
$[c,d] \neq 1$, then we can write the equation as
$$
[x,y]c^z = c^d = [d,c^{-1}]c
$$
which  has the  solution $x \rightarrow d, \ \ y \rightarrow c^{-1},
\ \  z \rightarrow 1$ such that $[x,y] \rightarrow [d,c^{-1}] \neq
1.$ So we can assume now that $[c,d] = 1,$ in which case we have the
equation
$$
[x,y]c^z = c \  \ \ or \ \  equivalently \ \ \ [x,y] = [c^{-1},z].
$$
The group $G$ is  a nonabelian CSA-group; hence the center of $G$ is
trivial. In particular, there exists an element $h \in G$ such that
$[c,h] \neq 1$. We see that $x \rightarrow c^{-1}, \  y \rightarrow
h, \  z \rightarrow h$ is a solution $\phi$ for which $[x,y]^\phi
\neq 1.$

Thus we have the equation $[x,y]c^z = [a,b]c^d,$ where $[a,b] \neq
1.$ Let $H=\langle G,t| [t,bc^d]=1\rangle$. Consider the map $\psi$ defined as follows:
$$
x^\psi = t^{-1}a, \ \ \ y^\psi = t^{-1}bt, \ \
z^\psi = dt.
$$
Straightforward computations show  that
$$
[x,y]^\psi = [a,b][b,t], \ \ and \ \ (c^z)^\psi = c^{dt};
$$
hence
$$
[x^\psi,y^\psi]c^{z^\psi}  = [a,b]c^d
$$
and consequently, $\psi$ is a solution.

We claim that $[r_1^\psi,r_2^\psi] \neq  1.$ Indeed, suppose
$[r_1^\psi,r_2^\psi] = 1;$ then  we have
$$
[[x,y]^\psi,c^{z^\psi}] = 1, \ [[a,b][b,t],c^{dt}] = 1, \
t^{-1}b^{-1}tb[b,a]t^{-1}d^{-1}c^{-1}dt[a,b]b^{-1}t^{-1}bd^{-1}cdt=1$$
which implies
$$t^{-1}b^{-1}tb[b,a]t^{-1}d^{-1}c^{-1}dt[a,b]b^{-1}bd^{-1}cd=1.$$

The letter $t$ disappears only if $c^d$ commutes with $b$ or $b^a$
commutes with $bc^d$. In both cases the last equality implies that
$[a,b]$ commutes with $c^d$  and $b$ commutes with $b^a$. Therefore
$[a,b]=1$ which contradicts to the choice of $a,b,c,d$.

 $\Box$

Now suppose that $m=1, n > 1$.  Let $\phi: G_S \longrightarrow G$ be
an arbitrary solution of $S = g$. Write
$$
h = g(\prod_{j = 3}^{n}c_j^{z_j})^{- \phi}
$$
and consider the equation \begin{equation} \label{h}
[x,y]c_1^{z_1}c_2^{z_2} = h. \end{equation} If this equation
satisfies the conclusion of the proposition ~\ref{prop5}, then by
induction the equation $S = g$ will satisfy the conclusion. So we
need to prove the proposition just for the equation (\ref{h}). There
are now  two possible cases.

Case a) There exists a solution $\xi$ of the equation (\ref{h}) such
that $(c_2^{z_2})^\xi \neq h.$ In this event by Lemma ~\ref{c2} the
equation
$$
[x,y]c_1^{z_1} = h(c_2^{z_2})^{-\xi } \neq 1
$$
has a solution $\theta$ in general position. Hence we can extend
this $\theta$ to a solution of (\ref{h})  in such a way that
$r_i^\theta \neq 1$ for $i = 1, 2$ and $[r_1^\theta,r_2^\theta] \neq
1$. Consequently, by Proposition ~\ref{prop1} we can construct a
solution $\psi$ in general position. It will automatically satisfy
the conclusion of Proposition \ref{prop5}.

Case (b) Assume now, that $(c_2^{z_2})^{\phi} = h$ for all solutions
$\phi$ of the equation (\ref{h}). Then we actually have
$$
[x,y]c_1^{z_1} = 1, \ \ and \ \  c_2^{z_2} = h,
$$
and this system of equations has  a solution in $G.$ It follows
that $c_1 = [a,b] \neq 1$ for some $a,b \in G$. Therefore the
equation (~\ref{h}) is in the form
$$
[x,y][a,b]^{z_1}c_2^{z_2} = h,
$$
and has a solution $\psi$ of the type
$$
x^\psi = b^f, \ y^\psi = a^f, \ z_1^\psi = f, \ z_2^\psi = z_2^\phi
$$
where  $f$ is an arbitrary element in $G$ and $\phi$ is an arbitrary
solution of (~\ref{h}).  The two elements $[a,b]$ and $h$ are
nontrivial in the CSA-group $G$
 hence  there exists an element $f^\ast \in G$ such that $[[a,b]^{f^\ast},h] \neq 1.$
But this implies that if we take $f = f^{\ast}$ then the solution
$\psi$ will have the property $[r_2^\psi,r_3^\psi] \neq 1.$ Now it
is sufficient to apply Proposition ~\ref{prop1}.

Now we suppose that $m = 2, n>1.$ In this event we have the equation
$$
[x_1,y_1][x_2,y_2] \prod_{j = 1}^{j = n} c_j^{z_j} = g.
$$
Again, if there exists  a solution $\phi$ of this equation such that
$$
(\prod_{j = 1}^{j = n} c_j^{z_j})^\phi  \neq  g,
$$
then we can write
$$
h = g(\prod_{j = 1}^{j = n} c_j^{z_j})^{-\phi},
$$
and consider the equation
$$
[x_1,y_1][x_2,y_2] = h
$$
which according to Lemma \ref{lemma3.5} has a solution $\xi$ in
general position such that the images of $x_i,y_i$ generate a free
subgroup. We can extend it to a solution of $S = g$ and by
Proposition \ref{prop5} applied to the equation
$$
[x_1^{\xi},y_1^{\xi}][x_2,y_2] \prod_{j = 1}^{j = n} c_j^{z_j} = g.
$$
we can construct a solution $\psi$ in general position with the
required properties.

Let assume now that
$$
(\prod_{j = 1}^{j = n} c_j^{z_j})^\phi  =  g
$$
for all solutions $\phi$ of the equation $S = g.$ This implies that
an arbitrary  map of the type
$$
x_1 \rightarrow a, \ y_1 \rightarrow b, \  x_2 \rightarrow b, \ y_2
\rightarrow a
$$
extends by means of any $\phi$ above to a solution $\psi$ of the
equation $S = g.$ Choose $a,b \in F$ then $[[b,a],r_3^\phi] \neq 1$
for the given solution $\phi.$  And we again just need to appeal to
Proposition \ref{prop5} for the equation
$$
[a,b][x_2,y_2] \prod_{j = 1}^{j = n} c_j^{z_j} = g.
$$

The case $m > 2$ is easy since if $\phi$ is a solution of the
equation $$\prod_{i = 1}^{i = m}[x_i,y_i] \prod_{j = 1}^{j =
n}c_j^{z_j}g^{-1}=1,$$  then we can consider the equation
$$
\prod_{i = 1}^{i = m} [x_i,y_i] = g(\prod_{j = 1}^{j = n}
c_j^{z_j})^{-\phi}
$$
which by Lemma \ref{lemma3.5} has a solution in general position
such that the images of $x_i,y_i$ generate a free subgroup; after
that to finish the proof we need only apply Proposition \ref{prop1}.

Proposition \ref{prop5} is proved.
\end{proof}

The following proposition settles genus 0 case.
\begin{prop}\label{prop6}
Let $S: c_1^{z_1} \ldots c_k^{z_k} = g$ be a nondegenerate standard
quadratic equation over a group $G\in\mathcal G$. Then either $S =
g$ has a solution in general position in some iterated centralizer
extension of $G\ast F$ or every solution of $S = g$ is commutative.
\end{prop}
{\em Proof.} By the definition of a standard quadratic equation $c_i
\neq 1$ for all $i = 1, \ldots ,k.$ Hence every solution of $S = g$
is a nondegenerate. Now the result follows from Proposition
~\ref{prop1}.

The following proposition can be proved similarly to Proposition 8
in \cite{KMNull}.
\begin{prop}
\label{neor} Let $S: x_1^2\ldots x_p^2c_1^{z_1}\ldots c_k^{z_k}g=1$
be a nondegenerate regular standard quadratic equation over a group
$G\in\mathcal G$. Then there is a solution in general position into
some iterated centralizer extension of $G\ast F$. If $p>2$ and
$p+k>3$, then the equation is regular.
\end{prop}

We introduce now some notation.
 For  $S: \prod_{i = 1}^{i =
m}[x_i,y_i] \prod_{j = 1}^{j = n}c_j^{z_j} = g,$ denote  $p_j=
c_j^{z_j}$, $ p_{n+1} =g^{-1},$ $q_k=\prod_{i = 1}^{i = k}[x_i,y_i]$
for $k\leq m$ and $q_{m+k}= \prod_{i = 1}^{i = m}[x_i,y_i] \prod_{j
= 1}^{j = k}p_k.$

For $S: \prod_{i = 1}^{i = m}x_i^2 \prod_{j = 1}^{j = n}c_j^{z_j} =
g,$ denote  $p_j= c_j^{z_j}$, $ p_{n+1} =g^{-1},$ $q_k=\prod_{i =
1}^{i = k}x_i^2$ for $k\leq m$ and $q_{m+k}= \prod_{i = 1}^{i =
m}x_i^2 \prod_{j = 1}^{j = k}p_k.$

\begin{prop}\label{propmain}
Let $S=g$ be a regular quadratic equation over a group $G\in\mathcal
G$. Then there exists a solution $\delta$ into  $G\ast F$ such that
for any $j=1,\ldots ,m+n-1$
\begin{enumerate}
\item $[q_j^{\delta},r_{j+1}^{\delta}]\neq 1;$
\item $[q_{j}^{\delta},(r_{j+1}\ldots r_{n+m})^{\delta}]\neq 1;$
\item There exists a solution $\delta$ into an iterated centralizer extension of $G\ast F$ such
that the following subgroups are free non-abelian: $\langle
q_j^{\delta},r_{j+1}^{\delta}\rangle $ for any $j=1,\ldots ,m+n-1;$
$\langle q_j^{\delta},x_{j+1}^{\delta}\rangle $ for any $j=1,\ldots
,m-1;$ $\langle q_{j+1}^{\delta},x_{j+1}^{\delta}\rangle $ for any
$j=1,\ldots ,m-1.$
\end{enumerate}\end{prop}

\begin{proof}
Let $S=g$ be an orientable equation. We begin with the first
statement. Let $\phi$ be a solution in general position constructed
in Proposition \ref{prop5}. Let $q_{j-1}=\prod
_{i=1}^{j-1}[x_i,y_i], A=q_{j-1}^{\phi}$, $x_{j}^{\phi}=a,
y_{j}^{\phi}=b,$ $x_{j+1}^{\phi}=c, y_{j+1}^{\phi}=d.$ If
$[A[a,b],[c,d]]\not =1$, then the statement is proved for $j$.
Suppose that $[A[a,b],[c,d]]=1.$ We can assume that $[b,c]\neq 1$
(taking $ab$ instead of $b$ if necessary). Let $t=bc^{-1}$. Take
another solution $\psi$ such that $q_{j-1}^{\psi}=q_{j-1}^{\phi},$
$x_{j}^{\psi}=t^{-s}a, y_{j}^{\psi}=b^{t^s},$
$x_{j+1}^{\psi}=c^{t^s}, y_{j+1}^{\psi}=t^{-s}d$ for a large $s\in
{\mathbb N}.$

If $[q_{j-1}^{\psi}[x_j^{\psi},y_j^{\psi}],
[x_{j+1}^{\psi},y_{j+1}^{\psi}]]=1$, then $$
A[a,b][b,t^s][t^s,c][c,d]=[t^s,c][c,d]A[a,b][b,t^s]$$ and,
therefore,
$$A[a,b][c,d]=[t^s,c]A[a,b][c,d][b,t^s].$$ If we denote
$B=A[a,b][c,d]$, this is equivalent to $B=[t^s,c]B[b,t^s]$ that is
equivalent, by commutation transitivity, to $[t,cBb^{-1}]=1$ or
$[t,B^{c^{-1}}]=1$, or $[B, c^{-1}b]=1$.

We take instead of $c,d$ respectively $(d^p)c, ((d^p)c)^kd$ and
denote the new solution by $\delta _{s,p,k}$. If $[q_j^{\delta
_{s,p,k}}, [x_{j+1}^{\delta _{s,p,k}}, y_{j+1}^{\delta
_{s,p,k}}]]=1$ for all $s,p,k$, then by the CSA property
$[b(d^pc)^{-1}, (d^pc)^kd]=1$ for all $p,k$, this contradicts to the
property that $c,d$ freely generate a free subgroup.

The proof for $j\geq m$ is similar.

The same solution $\delta _{s,p,k}$ can be used to prove the second
statement.

We will now prove the third statement by induction on $j$. Let
$\delta $ be a solution satisfying properties 1 and 2. Let $j=1$ and
$$H_1= \langle G*F,t_1|[t_1,(r_{2}\ldots
r_{m+n})^{\delta}]=1\rangle .$$ We transform $\delta$  into a
solution $\delta_1$ the following way. If $m\neq 0$, then
$$x_1^{\delta _1}=x_1^{\delta}, y_1^{\delta_1}=y_1^{\delta},$$ and
$$x_{i}^{\delta_1}=x_{i}^{\delta {t_1}}, y_{i}^{\delta_1}=y_{i}^{\delta
{t_1}}, z_k^{\delta_1}=z_k^{\delta }{t_1}$$ for $i=2,\ldots ,m,
k=1,\ldots, n.$ The subgroup generated by
$q_1^{\delta_1},r_{2}^{\delta_1}$, is free. Using Proposition
\ref{prop5} one can see that the subgroups generated by
$q_1^{\delta_1},x_{2}^{\delta_1}$ (if $m\geq 2$), and by
$q_2^{\delta_1},x_{2}^{\delta_1}$ are also free. In the case $m=0$
we define $$z_1^{\delta_1}=z_1^{\delta},\ z_k^{\delta_1}=z_k^{\delta
}t_1$$ for $i=2,\ldots ,m, k=1,\ldots, n.$

Suppose by induction that solution $\delta _{i-1}$ into a group
$H_{j-1}$ which is an iterated centralizer extension of $G*F$ and
satisfying the third statement of the proposition for indexes from 1
to $j-1$ has been constructed. Let
$$H_{j}=\langle H_{j-1},t_j|[t_j,(r_{j+1}\ldots
r_{m+n})^{\delta}]=1\rangle .$$

We begin with the  solution $\delta _{j-1}$ and transform it into a
solution $\delta_{j}$ the following way:
$$x_i^{\delta_j}=x_i^{\delta_{j-1}}, y_i^{\delta_j}=y_i^{\delta_{j-1}} ,\ i=1,\ldots ,j;$$ and
$$x_{i}^{\delta_j}=x_{i}^{\delta_{j-1} t_j}, y_{i}^{\delta_{j}}=y_{i}^{\delta_{j-1} t_j}$$
 for $i=j+1,\ldots ,m$,
$$z_i^{\delta _j}=z_i^{\delta_{j-1} }t_j.$$ The subgroups
generated by $q_j^{\delta_j},r_{j+1}^{\delta_j}$, by
$q_j^{\delta_j},x_{j+1}^{\delta_j}$ and by
$q_{j+1}^{\delta_j},x_{j+1}^{\delta_j}$ are free.

The proof for a non-orientable equation is very similar and we skip
it.

\end{proof}

We can now prove Theorem \ref{quadr}. Let $H$ be the fundamental
group of the graph of groups with two vertices, $v$ and $w$ such
that $v$ is a QH vertex, $H_w=\Gamma \in{\mathcal G},$ and there is
a retract from $H$ onto $\Gamma$.  Let $S_Q$ be a punctured surface
corresponding to a QH vertex group in this decomposition of $H$.
Elements $q_j, x_j$ correspond to simple closed curves on the
surface $S_Q$. By Proposition \ref{propmain}, we found a collection
of simple closed curves on $S_{Q}$ and  solution $\delta$ with the
properties 1)-4) from the beginning of Section 4.

Theorem E now follows from Theorem \ref{quadr} by induction.

Notice, that Proposition \ref{propmain} implies also the following
\begin{cor} (Compare to Lemma 1.32 \cite{S1}) Let $Q$ be a fundamental group of a punctured surface
$S_Q$ of Euler characteristic at most -2. Let $\mu :Q\rightarrow
\Gamma$ be a homomorphism that maps $Q$ into a non-abelian subgroup
of $\Gamma$ and the image of every boundary component of $Q$ is
non-trivial. Then either:
\begin{enumerate} \item there exists a separating s.c.c $\gamma\subset
S_Q$ such that $\gamma$ is mapped non-trivially into $\Gamma$, and
the image in $\Gamma$ of the fundamental group of each connected
components obtained by cutting $S_Q$ along $\gamma$ is non-abelian.
\item there exists a non-separating s.c.c. $\gamma\subset S_Q$ such
that $\gamma$ is mapped non-trivially into $\Gamma$, and the image
of the fundamental group of the connected component obtained by
cutting $S_Q$ along $\gamma$ is non-abelian.\end{enumerate}\end{cor}

In conclusion, we thank D. Osin who suggested a proof of Proposition \ref{bigpowers} and made
other useful comments.

\end{document}